\documentclass[12pt]{article}
\usepackage{amsmath,amsthm,amssymb,euscript,verbatim, enumerate}
\newtheorem{theorem}{Theorem}
\newtheorem{prop}[theorem]{Proposition}
\newtheorem{lemma}[theorem]{Lemma}

\theoremstyle{definition}
\newtheorem{defn}{Definition}
\newtheorem{remark}{Remark}
\newtheorem{corollary}{Corollary}

\newcommand{\F}{{F}}

\newcommand{\PP}{\mathbb P}
\newcommand{\R}{\mathbb R}
\newcommand{\C}{\mathbb C}
\newcommand{\CP}{\C\PP}
\newcommand{\CH}{\C{\mathrm H}}

\newcommand{\Lie}{{\mathcal L}}
\newcommand{\mean}{{\mathbf m}}
\def\({\left(}
\def\){\right)}
\def\<{\left <}
\def\>{\right >}
\def\a {\alpha}
\def\b {\beta}
\def\l {\lambda}

\newcommand{\I}{{\mathcal I}}

\newcommand{\V}{\mathcal V}

\renewcommand{\H}{\mathcal H}

\newcommand{\w}{\omega}

\newcommand{\e}{\mathbf e}

\newcommand{\bv}{\mathbf v}

\newcommand{\bw}{\mathbf w}

\newcommand{\JJ}{{\mathrm J}} 
\newcommand{\RR}{{\sf R}} 

\newcommand{\restr}{\negthickspace \mid}

\newcommand{\nat}{\widetilde\nabla}
\newcommand{\tr}{\operatorname{trace}\,}
\newcommand{\RRt}{\widetilde \RR}
\newcommand{\mt}{\widetilde M}
\newcommand{\Mt}{\widetilde M}
\def\intprod{\mathbin{\raisebox{.4ex}{\hbox{\vrule height .5pt width
5pt depth 0pt %
         \vrule height 3pt width .5pt depth 0pt}}}}

\def\&{\wedge}

\newcommand{\tmod}{\ \operatorname{mod}\, }

\def\aside#1{\noindent{\tt {#1}}\par\noindent}

\begin{document}
\title{The *-Ricci tensor for hypersurfaces in $\CP^n$ and $\CH^n$}
\author{Thomas A. Ivey and Patrick J. Ryan}
\date{April 2010}
\maketitle
\begin{abstract}
We update and refine the work of T. Hamada concerning *-Einstein hypersurfaces in $\CP^n$ and $\CH^n$.
We also address existence questions using the methods of moving frames and exterior differential systems.
\end{abstract}

\section{Introduction}

The notion of *-Ricci tensor for an almost-Hermitian manifold was introduced by Tachibana \cite{tachibana}
in 1959 and later used (along with the related concept of *-Einstein) in work on the Goldberg conjecture (see,
for example, Oguro and Sekigawa \cite{oguro}).  These ideas also apply naturally to contact metric manifolds, and
in particular, to hypersurfaces in complex space forms, where they were introduced by T. Hamada \cite{hamadaTokyo}.
In this paper, we refine, clarify, and extend some of Hamada's work, specifically the classification of
*-Einstein hypersurfaces in complex space forms.  See, in particular, Theorem \ref{starEinsteinClassification}.

Takagi \cite{takagi1975}, for $\CP^n$, and
Montiel \cite{montiel}, for $\CH^n$,
catalogued a specific list of real hypersurfaces, which we call ``Takagi's list" and ``Montiel's list" in
\cite{nrsurvey}.  These are the homogeneous Hopf hypersurfaces.  They have constant principal curvatures and
every Hopf hypersurface with constant principal curvatures is an open subset of one of them.

Many theorems have been published characterizing
these lists or subsets of them.  For example, the pseudo-Einstein hypersurfaces, introduced by Kon \cite{kon},
form such a subset.  The same subset is characterized as the the set of Hopf hypersurfaces
satisfying a certain condition on the Ricci tensor (known as {\it pseudo-Ryan} in the literature).  This has been known
for $n \ge 3$ since 1990 (see Theorems 6.1, 6.2, and 6.30 of \cite{nrsurvey}).  In Theorem \ref{pseudoRyanEinstein},
we prove this result for $n=2$.  We also prove that the *-Einstein and pseudo-Ryan conditions are equivalent for Hopf
hypersurfaces when $n=2$, thus giving us three distinct characterizations of this class of hypersurfaces.

It would be of interest to find additional classes of hypersurfaces, that could
be ``nicely" characterized, but this seems to be a difficult problem.  In this paper, we establish the existence
of a family of non-Hopf pseudo-Ryan hypersurfaces in $\CP^2$ and $\CH^2$, and prove that (in contrast to the Hopf case),
the set of non-Hopf pseudo-Ryan hypersurfaces is disjoint from the set of non-Hopf *-Einstein hypersurfaces;
see Theorem \ref{nonempty} and Corollary \ref{pseudoRexist}.  We hope
that this result will lead to further refinements of these conditions that can be characterized geometrically.

In \S\ref{constancy}, we construct a family of Hopf hypersurfaces that are not *-Einstein, but satisfy
a weakened form of the *-Einstein condition.  These examples show that the constancy of the *-scalar curvature is an essential
assumption in the definition of the *-Einstein condition, unlike the situation in the definition of ``ordinary"
Einstein manifold.  Finally, as a further application of our methods, in \S\ref{constantpc} we provide a new construction
for the non-Hopf hypersurfaces in $\CH^2$ with constant principal curvatures which were classified
by Berndt and Diaz-Ramos \cite{BerndtDiaz}.

In what follows, all manifolds are assumed connected and all
manifolds and maps are assumed smooth $(C^\infty)$ unless stated otherwise.  Basic
notation and historical information for hypersurfaces in complex space forms
may be found in \cite{nrsurvey}.  For more on
moving frames and exterior differential systems, see the monograph
\cite{BCG3} or the textbook \cite{cfb}.

\subsection{Complex space forms and the *-Ricci tensor}

Throughout this paper, we will take the holomorphic sectional curvature of the complex
space form in question to be $4c$.  The curvature operator $\RRt$ of the
space form satisfies
\begin{equation}\label{ambientcurvature}
\RRt(X,Y) = c (X \wedge Y + \JJ X \wedge \JJ Y + 2\langle X, \JJ Y\rangle \JJ)
\end{equation}
for tangent vectors $X$ and $Y$ (cf. Theorem 1.1 in \cite{nrsurvey}),
where $X\wedge Y$ denotes the skew-adjoint operator defined by
$$(X \wedge Y) Z = \langle Y, Z\rangle X - \langle X, Z\rangle Y.$$
We will denote by $r$ the positive number such
that $c = \pm 1/r^2$.  This is the same convention as used in (\cite{nrsurvey}, p. 237).

A real hypersurface $M$ in $\CP^n$ or $\CH^n$ inherits two
structures from the ambient space.  First, given a unit normal
$\xi$, the {\em structure vector field} $W$ on $M$ is defined so
that
$$\JJ W = \xi,$$
where $\JJ$ is the complex structure.
This gives an orthogonal splitting of the tangent space as
$$ \operatorname{span} \{W\} \oplus W^\perp.$$
Second, we define on $M$ the skew-symmetric
 $(1, 1)$ tensor field $\varphi$ which is the complex
structure $\JJ$ followed by projection, so that
$$\varphi X = \JJ X - \langle X,W\rangle \xi.$$

Recall that the type (1,1) Ricci tensor of any Riemannian manifold is defined by the equation
\begin{equation}
\<SX, Y\> = \text{trace\ } \ \{Z \mapsto R(Z, X) Y\}
\end{equation}
where $X$, $Y$, and $Z$ are any tangent vectors and $R$ is the curvature tensor.  In case of a K\"ahler manifold, it is not difficult
to show that
\begin{equation}
\<SX, Y\> = \tfrac{1}{2} (\text{trace\ } \{\JJ \circ R(X, \JJ Y)\}).
\end{equation}
(see \cite {KN}, p. 149).  This led Tachibana and others to consider, on any almost-Hermitian manifold,
the *-Ricci tensor $S^*$, which may be defined by the same formula,
\begin{equation}
\<S^* X, Y\> = \tfrac{1}{2} (\text{trace\ } \{\JJ \circ R(X, \JJ Y)\})
\end{equation}
and to define a space to be *-Einstein if $\<S^* X, Y\>$ is a constant multiple of $\< X, Y\>$ for all
tangent vector fields $X$ and $Y$.

\section{Basic equations for hypersurfaces}
\label{basic}

In this and subsequent sections, we follow the notation and terminology of \cite{nrsurvey}:
$M^{2n-1}$ will be a hypersurface
in a complex space form $\mt$ (either $\CP^n$ or $\CH^n$) having constant holomorphic sectional curvature $4c\ne 0$.
The structures $\xi$, $W$, and $\varphi$ are
as defined in the Introduction. The $(2n-2)$-dimensional
distribution $W^\perp$ is called the {\it holomorphic distribution}. The operator $\varphi$ annihilates
$W$ and acts as complex structure on $W^\perp$. The shape
operator $A$ is defined by
$$A X = -\nat_X \xi$$
where $\nat$ is the Levi-Civita connection of the ambient space.
The Gauss equation expresses the curvature operator of $M$ in terms
of $A$ and $\varphi$, as follows:
\begin{equation}\label{gausseq}
\RR (X,Y)=AX\& AY+c\(X\& Y+\varphi X\& \varphi Y +2 \<X,\varphi Y\>
\varphi\),
\end{equation}
and from this we see that the
Ricci tensor is given by
\begin{equation}\label{Ricci}
SX=(2n+1)cX-3c\langle X,W\rangle W+ \mean AX-A^2 X,
\end{equation}
 where $\mean = {\text{trace}}\ A$.
In addition, it is easy to show (see \cite{nrsurvey}, p. 239) that
\begin{equation} \label{nablaW}
\nabla_X W=\varphi A X,
\end{equation}
where $\nabla$ is the Levi-Civita connection of the hypersurface $M$.

Following Hamada \cite{hamadaTokyo}, we define the *-Ricci tensor $S^*$ on $M$ by
\begin{equation}
\<S^* X, Y\> = \tfrac{1}{2} (\text{trace\ } \{\varphi \circ \RR(X, \varphi Y)\}),
\end{equation}
and the *-scalar curvature $\rho^*$ to be the trace of $S^*$.  We say that the hypersurface $M$ is *-Einstein if
$\rho^*$ is constant and
\begin{equation} \label{starEinsteinCondition}
\<S^* X, Y\> = \frac{\rho^*}{2(n-1)} \< X, Y\>
\end{equation}
for all $X$ and $Y$ in the holomorphic distribution $W^{\perp}$.

We define
the function
$$\alpha = \langle A\,W, W\rangle.
$$
The hypersurface  is said to be Hopf if the structure vector $W$ is a principal vector,
i.e. $AW = \alpha W$, and we refer to
$\alpha$ as the {\it Hopf principal curvature}.  It is important to recall that the Hopf
principal curvature is constant (see Theorem 2.1 in \cite{nrsurvey}).
Of course, $\alpha$ need not be constant for a non-Hopf hypersurface.

We also recall the notion of {\em pseudo-Einstein hypersurface}.
A real hypersurface $M$ in $\CP^n$ or $\CH^n$ is said to be
pseudo-Einstein if there are constants $\rho$ and $\sigma$ such that
$$S X = \rho X + \sigma \langle X,W\rangle W$$
for all tangent vectors $X$.

\subsection{*-Einstein hypersurfaces in the Takagi and Montiel lists}

We first note which hypersurfaces in the Takagi and Montiel lists are *-Einstein.
According to the standard terminology (see, for example \cite{nrsurvey}, pp.254--262),
the lists are broken down into ``types" A1, A2, A0, B, C, D, and E.
The situation is as follows:


\begin{theorem} \label{TakagiList}
Among the  homogeneous Hopf hypersurfaces $M^{2n-1}$ in $\CP^n$ and $\CH^n$, where $n \ge 2$
(i.e. Takagi's and Montiel's lists),
\begin{itemize}
\item All type A1, A0 and B hypersurfaces are *-Einstein,
\item A type A2 hypersurface is *-Einstein if an only if it is
a tube of radius $\frac{\pi}{4}r$ over $\CP^k$ where $1 \le k \le
n-2$,
\item No type C, D, or E hypersurface is *-Einstein.
\end{itemize}
\end{theorem}
In other words, geodesic spheres in $\CP^n$, geodesic spheres,
horospheres, and tubes over $\CH^{n-1}$ in $\CH^n$ are *-Einstein, but except
for that, there is just one special case.  Note also, that the same classification holds locally. In
other words, an open subset of a hypersurface $M$ in the Takagi/Montiel lists is *-Einstein if and
only if $M$ is.

Theorem \ref{TakagiList} can be proved in a routine manner once we collect and verify a few facts.
We will do this at the end of Section  \ref{HopfStarEinstein}.

\subsection{Computation of the *-Ricci tensor}

In this section, we derive an expression for the *-Ricci tensor of a hypersurface and discuss the implications
for Hopf hypersurfaces.

\begin{theorem} \label{StarRicciTheorem}
For a real hypersurface $M^{2n-1}$ in $\CP^n$ or $\CH^n$, where $n \ge 2$,

\begin{equation}\label{Sstarform}
S^* = - (2nc\varphi^2 + (\varphi A)^2).
\end{equation}
Furthermore,
\begin{itemize}
\item If $M$ is Hopf, then $S^*$ is symmetric and $S^*W = 0$.
\item If $M$ is Hopf and $\a = 0$, then $S^*X = (2n+1)c X$ for all $X\in W^\perp$, and $\rho^* = 2(n-1)(2n+1)c$.
In particular, $M$ is *-Einstein.
\end{itemize}
\end{theorem}
\begin{proof}
We recall that for any linear functional $\psi$
on a finite-dimensional vector space, the trace of the map
$$
v \mapsto \psi(v) u
$$
is $\psi(u)$.
When we use the Gauss equation \eqref{gausseq} to compute $\RR(X,\varphi Y)\varphi Z$,
the first term is
\begin{equation}
(AX\& A\varphi Y) \varphi Z = \<A \varphi Y, \varphi Z\> A X - \<A X, \varphi Z\> A \varphi Y,
\end{equation}
so that
$$\tr (AX\& A\varphi Y)\circ \varphi =
\<A \varphi Y, \varphi A X\> - \<A X, \varphi A \varphi Y\> = -2 \<(\varphi A)^2 X, Y\>.
$$
Similarly, the other terms in the Gauss equation give
\begin{multline*}
(X\& \varphi Y+ \varphi X\& \varphi^2 Y +2 \<X,\varphi^2 Y\>\varphi ) \varphi Z = \\
\< \varphi Y, \varphi Z\>  X - \<X, \varphi Z\>  \varphi Y
+ \< \varphi^2 Y, \varphi Z\> \varphi X - \<\varphi X, \varphi Z\>  \varphi^2 Y
+ 2 \< X, \varphi^2 Y \> \varphi^2 Z
\end{multline*}
so that
\begin{multline*}
\tr
c(X\& \varphi Y+ \varphi X\& \varphi^2 Y +2 \<X,\varphi^2 Y\>\varphi )\circ \varphi =
\\
c(\< \varphi Y, \varphi X\> - \<X, \varphi^2 Y\>
+ \< \varphi^2 Y, \varphi^2 X\> - \<\varphi X, \varphi^3 Y\>  + 2 \< X, \varphi^2Y\> \text{\ trace\ }\varphi^2).
\end{multline*}
Noting that $\varphi^4 = -\varphi^2$ and $\operatorname{trace} \varphi^2 =-2(n-1)$, we find that
\begin{equation}
\<S^*X, Y\> = - \<(2nc\varphi^2 + (\varphi A)^2)X, Y\>.
\end{equation}

Now it is clear that $S^*$ is symmetric if and only if $(\varphi A)^2= (A \varphi)^2$.
In case $M$ is Hopf, we make use of the identity (\cite{nrsurvey} p. 245)
\newline
\begin{equation}\label{basicHopf}
A\varphi A = \frac{\a}2 (A\varphi +\varphi A) + c\varphi
\end{equation}
to reduce this condition to $ \frac{\a}2 (A\varphi^2)= \frac{\a}2 (\varphi^2 A)$.
Since $\operatorname{span} \{W\}$ and  $ W^\perp$ are $A$-invariant, we can use the fact that $\varphi^2$ is
zero on $W$ and acts as  $-I$ on $ W^\perp$ to verify that $A\varphi^2= \varphi^2 A$, and hence conclude
that $S^*$ is symmetric.

Finally, since $\varphi A W= 0$ for a Hopf hypersurface, we have $S^* W = 0$.  Further,
if $\a=0$, then applying $\varphi$ to \eqref{basicHopf} shows that $(\varphi A)^2 X = -c X$
for all $X \in W^\perp$. This yields the desired results for $S^*$ and $\rho^*$.
\end{proof}

\section{*-Einstein Hopf hypersurfaces}
\label{HopfStarEinstein}

In this section, we discuss the converse of Theorem \ref{TakagiList}.  Must every *-Einstein Hopf hypersurface occur in the
lists of Takagi and Montiel?  The answer is no, but almost.  Specifically, we have,
\begin{theorem}\label{nonzeroAlpha}
Let $M^{2n-1}$, where $n \ge 2$, be a *-Einstein Hopf hypersurface in $\CP^n$ or $\CH^n$ whose Hopf principal
curvature $\alpha$ is nonzero.
Then $M$ is an open subset of a hypersurface in the lists of Takagi and Montiel.
\end{theorem}
\begin{remark}
This corrects Theorems 1.1 and 1.2 of \cite{hamadaTokyo}, where the case $\alpha = 0$ was overlooked.  We will
show that all Hopf hypersurfaces with $\a = 0$ are *-Einstein.  In $\CP^n$, for instance, this includes every
hypersurface that is a tube of radius $\frac{\pi}{4}r$ over a complex submanifold.
Also, all pseudo-Einstein hypersurfaces in $\CP^2$ and $\CH^2$ are *-Einstein.  Many of these have non-constant
principal curvatures; see \cite{kimryan2} and \cite{IveyRyan}.
\end{remark}

We now prove Theorem \ref{nonzeroAlpha}.
\begin{proof}
For any unit principal vector  $X \in W^\perp$ with corresponding principal curvature $\l$,
it follows directly from \eqref{basicHopf} that
$(\l - \frac{\a}{2}) A \varphi X = (\frac{\l\a}{2} + c) \varphi X$.  If $\l \ne \frac{\a}{2}$, then
$\varphi X$  is also a principal
vector with corresponding principal curvature $\nu$ where
\newline
\begin{equation}\label{basicHopf1}
\l \nu = \frac {\l + \nu}{2} \a + c.
\end{equation}
We also note that $\frac{\a}{2}$ cannot be a principal curvature unless $\a^2 + 4c = 0$.

First look at the case where $\a^2 + 4c \ne 0$.
Pick a point  $p \in M$ where a maximal number of eigenvalues of $A$ (restricted
to $W^\perp$) are distinct.  This guarantees that the principal curvatures have constant multiplicities
in a neighborhood of $p$, and are therefore smooth.
Let $V \subseteq T_p M$ be a principal space corresponding to a principal curvature
$\l$.  Then $\varphi V$ is a principal space with corresponding principal curvature $\nu$ satisfying \eqref{basicHopf1}.
If $\operatorname{span} \{V, \varphi V\} = W^\perp$ at $p$, then $(\varphi A)^2 X = - \l \nu X$ for all $X \in W^\perp$.
Since $M$ is *-Einstein, $\l \nu$ must be constant near $p$.  Since $\a \ne 0$, we see from \eqref{basicHopf1}
that $\l + \nu$ is constant as well.  Thus $\l$ and $\nu$ are constant near $p$.  Note that this includes the case $\l = \nu$.
On the other hand, if $\operatorname{span} \{V, \varphi V\} \ne W^\perp$ at $p$,
we can construct (at least) two such such pairs
$\{\l, \nu\}$ and $\{\tilde\l, \tilde\nu\}$.  However, the *-Einstein condition guarantees that $(\varphi A)^2$
is a constant multiple of the identity on $W^\perp$, so that $\l \nu = \tilde \l \tilde \nu$, which leads to
$\l + \nu = \tilde \l + \tilde \nu$, and finally to $\{\l, \nu\} = \{\tilde\l, \tilde\nu\}$, which is a
contradiction.

We now consider the case where $\a^2 + 4c = 0$, and choose $p$ as above.  One possibility is that $AX = \frac{\a}{2}X$
for all $X \in W^\perp$.  If this does not hold, suppose that $\l \ne \frac{\a}{2}$ is a principal curvature at $p$, and that
$X$ is an associated principal vector.
Then $A \varphi X = \frac{\a}{2} \varphi X$ since (\ref{basicHopf1}) reduces to
\newline
\begin{equation} \label{basicHopf2}
(\l - \frac{\a}{2})(\nu - \frac{\a}{2}) = 0.
\end{equation}
Because $M$ is *-Einstein, $(\varphi A)^2 Y = -\l \frac{\a}{2} Y$ for all $Y \in W^\perp$ and $\l$ is constant
near $p$.  In particular, if
$AY = \frac{\a}{2}Y,$ this leads to $A \varphi Y = \l \varphi Y$.  Thus the principal
spaces of $\l$ and $\frac{\a}{2}$ have the same dimension and are interchanged by $\varphi$.
Because of (\ref{basicHopf2}), there can be no principal curvatures other than $\l$ and $\frac{\a}{2}$.  Again,
$M$ has constant principal curvatures near $p$.

In all cases, $p$ has a neighborhood with constant principal curvatures, which therefore must be an
open subset of some member of Takagi's or Montiel's lists.  For the case  $\a^2 + 4c = 0$
the only possibility that can actually occur is the horosphere and only $\a/2$ occurs as a principal
curvature on $W^\perp$.  Thus $M$ is an open subset of a horosphere.  For the case $\a^2 + 4c \ne 0$,
the set where the principal curvature data (value and multiplicity) agree with those at $p$, is open and
closed and therefore is all of $M$.  It follows that $M$ is an open subset of a specific member of one of these lists.

\end{proof}
As a consequence of Theorems \ref{StarRicciTheorem} and \ref{nonzeroAlpha}, Proposition 2.21 of \cite{kimryan2} and
Theorem 4 of \cite{IveyRyan2}, we have the following:
\begin{corollary}\label{StarRicciCorollary}
For a hypersurface $M^3$ in $\CP^2$ or $\CH^2$, the following are equivalent.
\begin{enumerate}
\item $M$ is Hopf and *-Einstein;
\item $M$ is pseudo-Einstein;
\item $\Lie_W R_W = 0$ where $R_W$ is the structure Jacobi operator of $M$ and $\Lie_W$ is the Lie derivative
in the direction of the structure vector $W$.
\end{enumerate}
\end{corollary}

\begin{remark}
The argument given in \cite{hamadaTokyo} for Theorem \ref{nonzeroAlpha} begins as ours does,
but leads to a quadratic equation with constant
coefficients, that all principal curvatures on $W^\perp$ must satisfy.  Unfortunately, when $\alpha=0$, all
coefficients vanish, so the proof is valid only when $\a \ne 0$ is assumed. We have included our alternative proof
in this paper because it establishes some facts that are useful for veryifying Theorem \ref{TakagiList},
to which we now turn our attention.
\end{remark}
\subsubsection*{Proof of Theorem \ref{TakagiList}}

As we have seen in the proof of Theorem \ref{nonzeroAlpha}, one can check the *-Einstein condition on a Hopf
hypersurface by examining the $\varphi$-invariant subspaces
of the form $\operatorname{span} \{V, \varphi V\}$, where $V\subseteq W^\perp$ is a principal subspace.
In the Takagi-Montiel lists, the type A hypersurfaces have $\varphi$-invariant principal spaces.  This
means that $(\varphi A)^2$, restricted to $W^\perp$, is a constant multiple
of the identity for type A1 and type A0 hypersurfaces,
so they must be *-Einstein.
For type A2 hypersurfaces, however, $W^\perp$ splits into two distinct $\varphi$-invariant principal subspaces,
whose corresponding principal curvatures ($\l_1$ and $\l_2$, say) satisfy the quadratic equation
$\l^2 = \a\l + c$.  In order to satisfy the *-Einstein condition, we would need to have $\l_1^2 = \l_2^2$,
which is impossible unless $\a = 0$.  For type A2 hypersurfaces in $\CP^n$, $\a = 0$ only when the radius is
$\frac{\pi}{4}r$, while for a type A2 hypersurface in $\CH^n$, $\a$ is nonzero for all radii.
(See Theorems 3.9 and 3.14 in \cite{nrsurvey}).

For type B hypersurfaces, $W^\perp = \operatorname{span} \{V, \varphi V\}$, where $V$ is a principal subspace
of dimension $(n-1)$.  The corresponding principal curvatures satisfy $\l_1 \l_2 + c = 0$, so that
$(\varphi A)^2 X = c X$ for all $X \in W^\perp$ and hence $M$ is *-Einstein with $\rho^* = 2(n-1)(2n-1)c$.

However, types C, D, and E hypersurfaces cannot be *-Einstein.  To see this, using the notation of \cite{nrsurvey}, p. 261,
we first note that $\a$ cannot be $0$ since $u = \frac{\pi}{4}$ would cause the principal curvature $\l_2$ to be undefined.
Further, principal curvatures $\l_1$ and $\l_3$ satisfy the quadratic equation $\l^2 = \l \a + c$, and hence the
corresponding principal spaces are $\varphi$-invariant.  The *-Einstein condition would then require that $\l_1^2 = \l_3^2$
which cannot be true since $\a \ne 0$.

Thus we have verified Theorem \ref{TakagiList}.  We now summarize the classification of *-Einstein Hopf hypersurfaces
as follows:
\begin{theorem}\label{starEinsteinClassification}
The *-Einstein Hopf hypersurfaces in $\CP^n$ and $\CH^n$, where $n \ge 2$, are precisely
\begin{enumerate}[(i)]
\item
the Hopf hypersurfaces whose Hopf principal curvature $\alpha$ vanishes, and
\item
the open connected subsets of homogeneous Hopf hypersurfaces of types A0, A1 and B.
\end{enumerate}
\end{theorem}
\begin{remark}
Note that geodesic spheres (type A1) of radius $\frac{\pi}{4}r$ in $\CP^n$ have $\alpha = 0$ and thus
satisfy both (i) and (ii) in Theorem \ref{starEinsteinClassification}.  Other than that, there is no overlap.
For further detail on
the structure of Hopf hypersurfaces with $\alpha=0$, see \cite{cecilryan} and \cite{IveyRyan}.
\end{remark}

\begin{corollary}\label{cor2}
For a Hopf hypersurface $M^{2n-1}$ in $\CP^n$ or $\CH^n$, where $n \ge 2$, $(\varphi A)^2$ cannot vanish identically.
\end{corollary}
\begin{proof}
Suppose that $(\varphi A)^2 = 0$.  Then $M$ is *-Einstein from Theorem \ref{StarRicciTheorem}.
If $\a = 0$, the result is immediate from \eqref{basicHopf}.  Otherwise, Theorem \ref{nonzeroAlpha} says
that $M$ must occur in the lists of Takagi or Montiel.  However, none of
the principal curvatures of these hypersurfaces (for principal spaces in $W^\perp$) vanish.
In fact, they all satisfy identities of the
form $\l^2 = \a\l + c\ $ or $\l\nu + c = 0$, so that the eigenvalues of $(\varphi A)^2$ on $W^\perp$ are all of
the form  $-\l^2 \ne 0$ or $-\l\nu = c \ne 0$.
\end{proof}

\section{Conditions on the Ricci tensor}\label{RicciConditions}
We recall the notation $R(X,Y)\cdot T$ for the action of a curvature operator on any tensor field $T$
(see \cite{nrsurvey} p. 235).  For the special case of the Ricci tensor $S$,
$$R(X,Y)\cdot S = R(X,Y) \circ S - S\circ R(X,Y).$$
For a Hopf hypersurface $M^{2n-1}$, where $n \ge 3$, the pseudo-Einstein condition is known to be
equivalent to the following:
\begin{equation}\label{pseudo-R}
\<(\RR(X_1, X_2)\cdot S)X_3, X_4\> = 0
\end{equation}
for all $X_1, X_2, X_3$ and $X_4$ in $W^\perp$ (see Theorem 6.30 of \cite{nrsurvey}).  A hypersurface
satisfying \eqref{pseudo-R} is called ``pseudo-Ryan" in the literature.  We now discuss this condition
for $n=2$ and how it relates to the pseudo-Einstein and *-Einstein conditions.

With this in mind, let $M^3$ be a (not necessarily Hopf) hypersurface in $\CP^2$ or $\CH^2$.
Suppose that there is a point $p$ (and hence an open neighborhood of $p$) where $AW \ne \alpha W$.
Then there is a positive function $\beta$ and a unit vector field $X \in W^\perp$ such that
$$AW = \alpha W + \beta X.$$
Let $Y = \varphi X$.  Then there are smooth functions $\lambda$, $\mu$, and $\nu$ defined near $p$ such
that with respect to the orthonormal frame $(W,X,Y)$,
\begin{equation} A =
 \begin {pmatrix} \a&\b& 0 \\
                \b &\l & \mu \\
                            0&\mu&\nu
\end {pmatrix}.\label{shapematrix}
\end{equation}
Note that if, on the other hand, $M$ is Hopf, then there still exists an orthonormal
frame $(W,X,Y)$ near any point, such that $Y = \varphi X$ and \eqref{shapematrix} holds
with $\beta=0$; however, the choice of $(X,Y)$ is only unique up to rotation.

Using \eqref{Ricci}, we compute the matrix of the Ricci tensor $S$ with respect to this frame,
to get,
\begin{equation} \label{RicciMatrix}
S =
\begin{pmatrix} 2c+\a(\l+\nu)-\b^2&\nu \b& -\mu\b \\
                \nu\b &5c+\l(\nu+\a)-\b^2-\mu^2 & \mu\a \\
                          -\mu\b&\mu\a&5c+\nu(\l+\a)-\mu^2
\end{pmatrix}.
\end{equation}
It is easy to check that \eqref{pseudo-R} is satisfied if and only if
$(\RR(X, Y)\cdot S)X$ and $(\RR(X, Y)\cdot
S)Y$ are multiples of $W$.

\begin{prop} \label{pseudo-R-criteria}
With $X, Y, \b, \mu, \nu$, and $\l$ defined as
above, $(\RR(X, Y)\cdot S)X$ and $(\RR(X, Y)\cdot
S)Y$ are multiples of $W$
 if and only if
\begin{equation}\label{firstCRcond}
\mu(\b^2\nu -\a(4c + \l\nu - \mu^2))=0
\end{equation}
and
\begin{equation}\label{secondCRcond}
\b^2(\mu^2-\nu^2)=(4c + \l\nu -\mu^2)(\a(\l-\nu)-\b^2).
\end{equation}
This equivalence also holds at a point where $AW = \a W$, where we take $X$
to be any unit principal vector in $W^\perp$, $Y = \varphi X$, and $\b = \mu =
0$.
\end{prop}
\begin{proof} The $X$ and $Y$ components of $(\RR(X, Y)\cdot S)X$
must be computed.  The Gauss equation gives the matrix of the curvature operator as
\begin{equation} \label{curvatureMatrix}
\RR(X, Y) =
\begin{pmatrix} 0&\mu \b& \nu\b \\
                -\mu\b & 0 & \l\nu -\mu^2 + 4c \\
                          -\nu\b& -(\l\nu -\mu^2 + 4c )&0
\end{pmatrix}.
\end{equation}

From this and matrix of $S$, the calculation is straightforward.
Also, since the calculation is pointwise, the expressions \eqref{RicciMatrix} and \eqref{curvatureMatrix}
and the conclusion are
equally valid at a point where $AW = \a W$, when $X$ is taken to be any unit principal vector in $W^\perp$,
$ Y = \varphi X$, and $\b = \mu = 0$.
\end{proof}

In particular, for a Hopf hypersurface, we have
\begin{theorem} \label{pseudoRyanEinstein}
A Hopf hypersurface $M^3$ in $\CP^2$ or $\CH^2$ is pseudo-Ryan if and
only if it is pseudo-Einstein.
\end{theorem}
\begin{proof}
We refer to the pointwise criterion for a Hopf hypersurface to be pseudo-Einstein, as given in
Proposition 2.21 of \cite{kimryan2}.
Let $p$ be any point of the Hopf hypersurface $M$ and let $X \in W^{\perp}$ be a unit
principal vector at $p$ with corresponding principal curvature $\l$.  Assume that \eqref{pseudo-R}
holds at $p$.  From \eqref{secondCRcond}, we have $\a(4c+\l\nu)(\l-\nu)=0.$

If $\a = 0$, then $\l\nu = c$ by \eqref{basicHopf1}, and the pseudo-Einstein criterion is
satisfied at $p$.  If $\a \ne 0$ and $\l \ne \nu$ at $p$, then $4c+\l\nu = 0$ near $p$.
Using \eqref{basicHopf1}, we get $-4c = \frac{\l+\nu}{2} \a +c$ so that $\l+\nu$ is also
constant.  Therefore, $\l$ and $\nu$ are constant and a neighborhood of $p$ is a
Hopf hypersurface with constant principal curvatures.
However, the well-known classification of such does not admit
this possibility.  As seen from Theorem 4.13 of \cite{nrsurvey}, such a hypersurface
would have to be in the Takagi/Montiel list and thus have $\l\nu+c = 0$.
Because of this contradiction, we conclude that $\l = \nu$ at $p$, and thus the pseudo-Einstein
criterion is satisfied there.  Since $p$ was arbitrary, $M$ must be pseudo-Einstein.

Conversely, if $M$ is pseudo-Einstein, then either $\a=0$ or $\l = \nu$ so that the equations
in Proposition \ref{pseudo-R-criteria} are satisfied, and $M$ is pseudo-Ryan.
\end{proof}

In view of our work in the previous section, we see that for Hopf hypersurfaces in $\CP^2$ and
$\CH^2$, the pseudo-Einstein, *-Einstein, and pseudo-Ryan conditions are equivalent.  We now look
at non-Hopf hypersurfaces.

First, we improve Proposition \ref{pseudo-R-criteria}.  Specifically, we
deduce that under the conditions of the proposition, we must have $\mu = 0$.
\begin{prop} \label{pseudo-R-refined}
In the notation of Proposition \ref{pseudo-R-criteria}, $(\RR(X, Y)\cdot S)X$ and $(\RR(X, Y)\cdot
S)Y$ are multiples of $W$
 if and only if $\mu = 0$,
and
$$ \b^2\nu^2= - (4c + \l\nu)(\a(\l-\nu)-\b^2).$$
\end{prop}
\begin{proof}
Suppose that \eqref{firstCRcond} and \eqref{secondCRcond} hold and $\mu \ne 0$ at some point.  Then,
in a neighborhood of this point, we have
\begin{equation}\label{betaduct}
\b^2\nu  = \a(4c + \l\nu - \mu^2).
\end{equation}
Multiplying \eqref{secondCRcond} by $\a$, we get
\begin{equation}\label{alphaduct}
\a\b^2(\mu^2-\nu^2)=\a (4c + \l\nu -\mu^2)(\a(\l-\nu)-\b^2),
\end{equation}
which, upon substitution from \eqref{betaduct}, yields
\begin{equation}
\a\b^2(\mu^2-\nu^2) = \b^2 \nu (\a(\l-\nu)-\b^2),
\end{equation}
Cancelling $\b^2$ and substituting for $\b^2 \nu$ from the first equation, we get
$4c\a = 0$.  Therefore $\a= \nu = 0$ and \eqref{alphaduct} reduces to
\begin{equation}\label{pseudoRcond}
\b^2\mu^2 = -(4c - \mu^2)\b^2,
\end{equation}
a contradiction.  We conclude that $\mu$ must vanish identically.  As the converse is trivial,
our proof is complete.
\end{proof}
We now turn our attention to the *-Einstein condition.
In a neighborhood of a point where $AW\ne\alpha W$, using the same orthonormal frame $(W, X, \varphi X)$,
it is easy to see from \eqref{shapematrix} that
\begin{equation} (\varphi A)^2 =
 \begin {pmatrix}\label{phiAsquared} 0&0& 0 \\
                -\b \nu &\mu^2-\l\nu & 0\\
                            \b \mu & 0 & \mu^2-\l\nu \\
\end {pmatrix}.
\end{equation}
Then, using \eqref{Sstarform}, we see that $\rho^*$ is locally constant if and only if $\mu^2-\l\nu$ is.
For a point where $AW = \a W$, we let $X$ be a unit principal vector in $W^\perp$ (as before) and let $Y = \varphi X$.
Then there are numbers $\a$, $\l$ and $\nu$ such that equations \eqref{shapematrix} and \eqref{phiAsquared}
still hold at this point, but with $\b = \mu =0$.

Although we do not wish to discuss ruled hypersurfaces in depth in this paper, they are useful
for demonstrating the non-equivalence of the pseudo-Ryan and *-Einstein conditions. To be concise,
we define a hypersurface in $\CP^n$ or $\CH^n$ to be {\it ruled} if
\begin{equation} \label{rrh}
A W^\perp \subseteq \operatorname{span} W.
\end{equation}
Geometrically, this means that $M$ is foliated by totally geodesic complex hypersurfaces (i.e. real codimension 2)
which are orthogonal to $W$.

\begin{prop}
For non-Hopf hypersurfaces $M^3$ in $\CP^2$ and $\CH^2$, the pseudo-Ryan and *-Einstein
conditions are not equivalent.  In fact,
\begin{itemize}
\item
All ruled hypersurfaces in $\CP^2$ and $\CH^2$ are *-Einstein;
\item
No pseudo-Ryan hypersurface in $\CP^2$ or $\CH^2$ is ruled.
\end{itemize}
\end{prop}
\begin{proof}
Clearly, a ruled hypersurface satisfies $(\varphi A)^2 = 0$ and hence is *-Einstein by Theorem \ref{StarRicciTheorem}
with $\rho^* = 2(n-1)(2n+1)c$.  This fact was observed by Hamada \cite{hamadaTokyo}.

It is also immediate from Corollary \ref{cor2} that no Hopf hypersurface can be ruled.  Therefore, any
ruled hypersurface in $\CP^2$ or $\CH^2$,
must have a point $p$ with a neighborhood in which the setup introduced at the beginning of this section holds
with $\b \ne 0$ and $\mu=\nu=\l=0$. The equations in Proposition \ref{pseudo-R-refined} reduce to $-4c \b^2= 0$,
a contradiction.  Thus no pseudo-Ryan hypersurface in $\CP^2$ or $\CH^2$ is ruled.
(It is easy to check that this also holds in $\CP^n$ and $\CH^n$ for $n \ge 3$).
\end{proof}

We now derive the conditions for a non-Hopf hypersurface to be {\it both} pseudo-Ryan and *-Einstein.
\begin{prop}\label{etcProp}
Let $M^3$ be a non-Hopf hypersurface $M^3$ in $\CP^2$ and $\CH^2$ that is both pseudo-Ryan and *-Einstein.
Then around any point where $AW \ne \a W$ there is
either

(i) a neighborhood in which
the components of the shape operator \eqref{shapematrix}
with respect to the standard basis satisfy
$\mu=\l=0$, $\b\nu \ne 0$, and
\begin{equation}\label{firstoption}
\b^2(\nu^2-4c)=4c\a \nu,
\end{equation}

or (ii) a neighborhood in which $\mu=0$, $\sigma=-\lambda\nu$ is a nonzero constant,
and
\begin{equation}\label{secondoption}
\b^2\nu^2=(4c - \sigma)\left(\a\left(\nu+\frac{\sigma}{\nu} \right)+\b^2\right).
\end{equation}
\end{prop}

\begin{proof}
Let $\sigma = - \l \nu.$
Since $M$ is pseudo-Ryan, by Proposition \ref{pseudo-R-refined} we have $\mu=0$ and
\begin{equation}\label{pseudoRcond2}
\beta^2 \nu^2 = - (4c-\sigma)(\alpha(\lambda-\nu) - \beta^2).
\end{equation}
\noindent
Since $M$ is *-Einstein, \eqref{phiAsquared} shows that $\sigma$ is locally constant.
There are two possibilities:
\begin{itemize}
\item If $\sigma=0$ then $\l$ must be identically zero.  Otherwise, there is an open set where $\nu = 0$ and $\a\l-\b^2=0$.
But this would require that  $\operatorname{rank}\ A \le 1$, contradicting a well-known fact about hypersurfaces
(see Proposition 2.14 of \cite{nrsurvey}).  Setting $\sigma=\lambda=0$ in \eqref{pseudoRcond2} yields \eqref{firstoption}.
\item
If $\sigma \ne 0$, then $\nu$ is nonvanishing, and setting $\lambda=-\sigma/\nu$ in \eqref{pseudoRcond2}
gives \eqref{secondoption}.
Note that the constant $4c-\sigma$ must be nonzero, since $\sigma=4c$ in \eqref{secondoption}
would imply that $\nu=0$.
\end{itemize}
\end{proof}

Based on the conditions derived in the previous proposition, we can deduce
\begin{prop} \label{nonHopfpseudoR}
If hypersurface $M^3$ is non-Hopf, pseudo-Ryan and *-Einstein, then (in a neighborhood of a point
where $AW \ne \alpha W$), we have $\sigma\ne 0$
and $\alpha,\beta,\lambda$ constant.
\end{prop}

This will be proved in Section \ref{movingframes} below.
For $\CH^2$, according to Berndt and Diaz-Ramos \cite{BerndtDiaz}, this implies that $M$ is one of the Berndt orbits -- either the
minimal orbit or one of its equidistant hypersurfaces. However, these hypersurfaces do not satisfy the conditions of Proposition
\ref{pseudo-R-refined}.  This can be seen from their principal curvatures which are given explicitly in Proposition 3.5 of \cite{BerndtDiaz}.
On the other hand, all hypersurfaces in $\CP^2$ with constant principal curvatures must be Hopf, as shown by Q.M. Wang \cite{qmwang}.
Thus, in fact, the kind of hypersurface envisioned in Proposition \ref{nonHopfpseudoR} does not exist and we have the following
improvement of Proposition 7:
\begin{theorem} \label{nonempty}
In $\CP^2$ and $\CH^2$ the set of non-Hopf *-Einstein hypersurfaces is disjoint from
the set of non-Hopf pseudo-Ryan hypersurfaces.
\end{theorem}

The examples we will construct in \S\ref{coho} (see Corollary \ref{pseudoRexist}) show that the set of non-Hopf pseudo-Ryan hypersurfaces is non-empty.

\section{Constancy of $\rho^*$}\label{constancy}

It is well-known that for a Riemannian manifold of dimension greater than 2, if $\<SX, Y\> = \rho \<X, Y\>$ for all
vector fields $X$ and $Y$, then $\rho$ must necessarily be constant.  One can ask similarly if,
in the definition \eqref{starEinsteinCondition} of *-Einstein, the stipulation
that $\rho^*$ be constant is redundant.

When $n = 2$, (i.e. for hypersurfaces of dimension 3), we find that the condition is not redundant.
In fact, using Theorem \ref{StarRicciTheorem},
\eqref{shapematrix}, and \eqref{phiAsquared}, we have the following:

\begin{prop}
Every hypersurface in $\CP^2$ or $\CH^2$ satisfies
\begin{equation*}
\<S^* X, Y\> = \frac{\rho^*}{2} \< X, Y\>
\end{equation*}
for all $X$ and $Y$ in $W^\perp$, with $\frac{\rho^*}{2} =4c+\lambda\nu - \mu^2$.
\end{prop}

In particular, for a Hopf hypersurface, we have $\frac{\rho^*}{2} =4c+\lambda\nu$.  For a Hopf hypersurface with
 $\alpha^2 +4c\ne0$, it follows from \eqref{basicHopf} that if $\rho^*$ were constant, then each of the principal curvatures $\lambda$ and $\nu$ would have to be constant.  Thus, we can obtain examples of Hopf hypersurfaces with nonconstant $\rho^*$ by constructing examples with non-constant principal curvatures associated to principal directions in $W^\perp$.
These are provided by tubes over holomorphic curves in $\CP^2$ or $\CH^2$;
in the latter case, $\alpha^2 + 4c>0$.  In both spaces, we can also construct
Hopf hypersurfaces with nonconstant principal curvatures using Theorem \ref{HopfIVP}
below, which allows us to prescribe the principal curvature along a principal curve perpendicular
to the structure vector.

\begin{remark}
Our result shows that the constancy of $\rho^*$ should be added to the hypotheses of Lemma 3.1
in Hamada's paper \cite {hamadaTokyo}.
\end{remark}

Before stating the theorem, we will introduce some necessary terminology
 for Frenet-type  invariants of curves in $\Mt=\CP^2$ and $\CH^2$, defined in terms of
unitary frames.\footnote{The usual (Riemannian) construction for Frenet frames along curves in these spaces, as set forth
in several papers by Maeda and collaborators \cite{MaedaCurve1}, \cite{MaedaCurve2},
is not suitable for our purposes, as we prefer to use frames adapted to the complex structure.}
 A {\em unitary frame} at a point $p\in \Mt$ is an orthonormal basis $(e_1, e_2, e_3, e_4)$
for $T_p\Mt$ such that
\begin{equation}
e_2 = \JJ e_1, \qquad e_4 = \JJ e_3.\label{unitarycond}
\end{equation}

\begin{defn}  Let $I\subset \R$ be an open interval, and $\gamma: I \to \Mt$ a regular unit-speed curve.
We say that
$\gamma$ is a regular {\em framed curve} if there exists a unitary
frame $(T, \JJ T, N, \JJ N)$ defined along $\gamma$ such that
$\gamma' = T$ and the frame vectors satisfy
\begin{equation}\label{Frenet}
T' = k_0 \JJ T + k_1 N, \qquad N' = -k_1 T + \tau \JJ N.
\end{equation}
where the primes indicate covariant derivative with respect to $\gamma'$ along $\gamma$,
and $k_0,k_1$ and $\tau$ are smooth functions referred to
as the {\em holomorphic curvature}, {\em transverse curvature} and {\em torsion}
respectively of $\gamma$.
\end{defn}

\begin{theorem}\label{HopfIVP}
Let $\gamma:I \hookrightarrow \Mt$ be a regular real-analytic framed curve with
zero torsion, and transverse curvature given by an analytic function $k_1(s)$.
For any real number $\alpha$ satisfying $\alpha^2+4c\ne0$, there exists a Hopf
hypersurface $M$ with Hopf principal curvature $\alpha$, containing $\gamma$,
and for which $\gamma$ is a principal curve perpendicular to the structure vector field $W$
with principal curvature $k_1$.  Any other Hopf hypersurface with these properties
will coincide with $M$ on an open set containing $\gamma$.
\end{theorem}
This result will be proved in \S\ref{Hopfmoving}.

The existence of a regular framed curve with $k_0,k_1$ and $\tau$ equal to any given smooth functions
can be shown by standard arguments about solutions of linear systems of ordinary differential equations (cf. Theorem 5.1 in \cite{MaedaCurve1}), and these
arguments carry over to the real-analytic category.  Thus, we can apply Theorem \ref{HopfIVP}
to produce a Hopf hypersurface with a zero torsion analytic curve as principal
curve, with \emph{any given analytic function} as principal curvature along this curve.

\section{Proofs using Exterior Differential Systems}\label{movingframes}

Let $\F$ be the unitary frame bundle of $\Mt=\CP^2$ or $\CH^2$, i.e., the bundle
whose fiber at a point $p$ is the set of
orthonormal frames $(e_1,e_2,e_3,e_4)$ satisfying \eqref{unitarycond}.
This is a principal sub-bundle of the full orthonormal frame bundle, and
has structure group $U(2)$.  Let $\w^i$ and $\w^i_j$, for $1\le i,j\le 4$, denote the
pullbacks of the canonical forms and Levi-Civita
connection forms from the full frame bundle.  If
$f=(e_1, e_2, e_3, e_4)$ is any local section of $\F$, then
the pullbacks of the $\w^i$ form a dual coframe, i.e.,
\begin{equation}\label{boilercanon}
e_j \intprod f^* \w^i = \delta^i_j.
\end{equation}
As well,
the connection forms have the property that
\begin{equation}\label{boilerconn}
\nat_\bv e_i = (\bv \intprod f^*\w^j_i)e_j
\end{equation}
for any tangent vector $\bv$.
These forms satisfy the usual structure equations
$$d\w^i = -\w^i_j \& \w^j, \qquad d\w^i_j = -\w^i_k \& \w^k_j + \Phi^i_j,$$
where $\Phi^i_j$ are the curvature 2-forms.
On the unitary frame bundle,
the connection forms $\w^i_j$ satisfy the additional linear relations
$$\w^3_1 = \w^4_2, \qquad \w^3_2 = -\w^4_1.$$
Thus, the canonical forms $\w^1, \w^2, \w^3, \w^4$ together with
the connection forms $\w^2_1, \w^4_1, \w^4_2, \w^4_3$ form a globally-defined coframe
on the 8-dimensional manifold $\F$.

The curvature forms are as follows:
\begin{align*}
\Phi^2_1 &= c (4 \w^2 \& \w^1 + 2 \w^4 \& \w^3), \\
\Phi^4_3 &= c( 4 \w^4 \& \w^3 +2\w^2 \& \w^1),\\
\Phi^4_1 &= \Phi^2_3 = c(\w^2 \& \w^3 + \w^4 \& \w^1), \\
\Phi^4_2 &= \Phi^3_1= c (\w^3 \& \w^1 + \w^4 \& \w^2). \\
\end{align*}

For a hypersurface $M \subset \Mt$, we say that a section $f:M \to \F\restr_M$
is an {\em adapted frame along $M$} if $e_4$ is normal to $M$.  It follows
from \eqref{boilercanon} that $f^* \w^4=0$, and it follows from \eqref{boilerconn}
that
\begin{equation}\label{omegaH}
f^* \w^4_i = h_{ij} f^* \w^j, \quad 1\le i,j\le 3,
\end{equation}
where $h_{ij}$ are the components of the shape operator with respect to
the tangential moving frame $(e_1, e_2, e_3)$.
Furthermore, if $(W,X,Y)$ is a local frame along $M$ as constructed
in \S\ref{RicciConditions}, then we may take $e_3= W$, $e_1=X$ and $e_2=Y$,
and the entries of $H$ are just the
entries of $A$ from \eqref{shapematrix} rearranged:
\begin{equation}\label{Hmatrix}
H = \begin{pmatrix} \lambda & \mu &\beta \\ \mu & \nu & 0 \\ \beta & 0 & \alpha \end{pmatrix}.
\end{equation}

\subsection{An initial value problem for Hopf hypersurfaces}\label{Hopfmoving}

From \eqref{boilerconn} and \eqref{Hmatrix}, it follows that an adapted unitary frame along
a Hopf hypersurface $M$ gives a section $f:M \to \F$ such that
$f^* \w^4 =0$ and $f^* (\w^4_3 - \alpha \w^3)=0$.  Thus, the image $\Sigma$
of $f$ is a three-dimensional submanifold along which the forms
$$\theta_1 = \w^4, \qquad \theta_2 = \w^4_3 - \alpha \w^3$$
pull back to be zero.  We let $\I$ be the exterior differential system
on $\F$ generated by these 1-forms (for a given value of the constant $\alpha$).
Then there is a one-to-one correspondence between Hopf hypersurfaces $M$
equipped with an adapted unitary frame and three-dimensional integral submanifolds of $\I$
satisfying the independence condition $\w^1\& \w^2 \& \w^3 \ne 0$.

To complete a set of algebraic generators for the ideal $\I$, we need to
compute the exterior derivatives
$$\begin{aligned} d\theta_1 &\equiv -\w^4_1 \& \w^1 -\w^4_2 \& \w^2,\\
d\theta_2 &\equiv 2(\w^4_1 \& \w^4_2 - c\w^1 \& \w^2) + \alpha (\w^4_2 \& \w^1 - \w^4_1 \& \w^2)
\end{aligned} \mod \theta_1, \theta_2.$$
Let $\Omega_1, \Omega_2$ be the 2-forms on the right hand sides of these equations.

For an exterior differential system, a subspace $E$ of the tangent space at a point in
the underlying manifold is an {\em integral element} if all differential
forms in the system restrict to be zero on $E$.   For example, let $u\in \F$ and let $\bv\in T_u \F$ be a nonzero vector; then $\bv$ spans a 1-dimensional integral element if and only
if $\bv\intprod \theta_1 = \bv\intprod \theta_2 = 0$.

We define the {\em polar space} of a $k$-dimensional integral element as follows.
\begin{defn}
Let $\{\e_1, \cdots, \e_k\}$ be a basis for an integral element $E$.  The {\em polar space}
$H(E)$ of $E$ is the set of all $\bw \in T_u \F$ such that $\psi(\bw, \e_1, \cdots, \e_k) = 0$
for all $(k+1)$-forms $\psi \in  \I$.
\end{defn}
In other words, $H(E)$ contains all possible enlargements of $E$ to
an integral element of one dimension larger (cf. \cite{cfb}, \S7.1)
For example, the polar space of the span of vector $\bv \in T_u \F$ is
\begin{equation}\label{polardef}
H(\bv) = \{ \theta_1, \theta_2, \bv \intprod \Omega_1, \bv\intprod \Omega_2\} ^\perp \subset T_u \F,
\end{equation}
where the $\perp$ sign indicates the subspace annihilated by the 1-forms in braces.

\begin{lemma}\label{Hlem} Let $\V^1_u \subset T_u \F$ be the set of vectors $\bv$ such that
$\bv\intprod \theta_1 = \bv\intprod \theta_2 = 0$ but $\bv\intprod (\w^1 \& \w^2 \& \w^3) \ne 0$.  Then
$\dim H(\bv)=4$ for an open dense subset of $\V^1_u$.
\end{lemma}

For $\bv \in \V^1_u$, we will say that $\bv$ is {\em characteristic} if $\dim H(\bv) > 4$.

\begin{proof}
Because $T_u F$ is 8-dimensional, the dimension of $H(\bv)$ is 8 minus the dimension of the
span of the 1-forms in braces in \eqref{polardef}.  This in turn depends only on the
values of $\bv \intprod \Omega_1$ and $\bv \intprod \Omega_2$.
Suppose that
\begin{equation}\label{vintprod}
\bv \intprod \begin{pmatrix}\w^4_1 \\ \w^4_2\\ \w^1 \\ \w^2  \end{pmatrix} = \begin{bmatrix} p \\ q\\ a\\ b  \end{bmatrix}.
\end{equation}
Then
$$
\begin{pmatrix} \bv \intprod \Omega_1 \\ \bv \intprod \Omega_2 \end{pmatrix}
= \begin{bmatrix} a & b & -p & -q \\ -2q+\alpha b & 2p-\alpha a & 2cb+\alpha q & -(2ca+\alpha p)\end{bmatrix}
\begin{pmatrix}\w^4_1 \\ \w^4_2\\ \w^1 \\ \w^2  \end{pmatrix}.
$$
Let $R$ be the $2\times 4$ matrix on the right-hand side.  It is easy to check that $R$ has rank 2
unless
\begin{equation}\label{CV} 0 = 2(ap+bq)-\alpha (a^2+b^2) = p^2+q^2+c(a^2+b^2).
\end{equation}
These equations fail to hold simultaneously on an open dense subset of $\V^1_u$.  For vectors in this set,
$\bv \intprod \Omega_1$
and $\bv \intprod \Omega_2$ are linearly independent combinations of $\w^4_1, \w^4_2, \w^1, \w^2$,
and hence $\dim H(\bv)=4$.
\end{proof}

\begin{remark}It is evident from the equations \eqref{CV} that when $c>0$ the only characteristic
vectors are those for which $a=b=p=q=0$, forming a 2-dimensional plane
in $T_u F$.  (The same is true if $c<0$ and $|\alpha|> 2/r$.)
However, when $c<0$ and $|\alpha|\le 2/r$
 the set of characteristic vectors is the union of the 2-dimensional plane
 and a 4-dimensional submanifold, parametrized by $a$ and $b$
(not both zero) and the values of $\bv \intprod \w^3$ and $\bv \intprod \w^2_1$.
\end{remark}

As stated above, a Hopf hypersurface with an adapted unitary frame corresponds
exactly to a 3-dimensional integral submanifold $\Sigma \subset \F$ along which
the independence condition $\w^1 \& \w^2 \& \w^3 \ne 0$ holds.  However, an adapted framing
$(e_1,\ldots, e_4)$ along a given Hopf hypersurface can always be modified by a
rotation
$e_1 \mapsto \cos \psi\ e_1 + \sin\psi\ e_2, e_2\mapsto -\sin\psi\ e_1 + \cos \psi\ e_2$
for an arbitrary angle $\psi$.  Under such changes, the corresponding
section of $\F\restr_M$ moves along circle within the fibers of $\F$.  On these circles the
1-form $\w^2_1$ restricts to be nonzero but the remaining 1-forms
$\w^1, \w^2, \w^3, \w^4, \w^4_1, \w^4_2, \w^4_3$ of the standard coframe
restrict to be zero.  Vector fields tangent to these circles
are {\em Cauchy characteristic vectors} for $\I$ (see \S6.1 in \cite{cfb} for
more information).  In particular, if $\Sigma^3$ is an integral
manifold satisfying $\w^1 \& \w^2 \& \w^3 \ne 0$, then $\Sigma$ is transverse
to the Cauchy characteristic circles, and the union of the
circles through $\Sigma$ is a 4-dimensional integral manifold of $\I$.
Thus, $M$ is associated to a {\em unique} 4-dimensional integral submanifold of $\I$ satisfying the independence
condition
\begin{equation}\w^1 \& \w^2 \& \w^3 \& \w^2_1 \ne 0.\label{4dindy}
\end{equation}

\begin{prop}\label{applyCK} Let $\alpha$ be any real number satisfying $\alpha^2+4c\ne 0$, and
let $\I$ be the exterior differential system
on $\F$ generated by $\theta_1$ and $\theta_2$.
Let $\Gamma: I \hookrightarrow \F$ be a real-analytic curve such that $\Gamma'(t) \in \V^1_{\Gamma(t)}$
for all $t \in I$ and $\Gamma'(t)$ is never characteristic.
Then there exists a unique real-analytic submanifold $\Sigma^3 \subset \F$ that contains
$\Gamma$ and is an integral submanifold of $\I$ satisfying the independence condition
\eqref{4dindy}.
\end{prop}

\begin{proof} We will apply the Cartan-K\"ahler Theorem and Cartan's Test for involutivity.
(For Cartan's Test, see Theorem 7.4.1 in \cite{cfb}; for Cartan-K\"ahler see
Theorem III.2.2 in \cite{BCG3}.)
This will first require investigating the equations that define
the set of 4-dimensional integral elements of $\I$.

At a point $u\in \F$, a 4-dimensional subspace $E\subset T_u \F$
is by definition an integral element of $\I$ if all differential forms in $\I$ restrict to be zero
on $E$.  (We will consider only those 4-planes $E$ that satisfy the independence condition \eqref{4dindy}.)
In order that the algebraic generators of $\I$ vanish on $E$, the restrictions of the 1-forms $\w^4, \w^4_1, \w^4_2, \w^4_3$ to $E$ must satisfy
\begin{equation}\label{omegavals}
\w^4=0,\quad \w^4_3 = \alpha \w^3,\quad
\w^4_1 = \lambda \w^1 + \mu \w^2, \quad \w^4_2 = \mu \w^1 +\nu  \w^2,
\end{equation}
for values of $\lambda, \mu, \nu$ such that
\begin{equation}\label{inteltq}
2(\lambda \nu-\mu^2 -c)-\alpha(\lambda+\nu)=0.
\end{equation}
This equation is obtained by substituting \eqref{omegavals} in $\Omega_2$.

Using the above equations, it is easy to check that the set of integral 4-planes
is a smooth 2-dimensional submanifold of the Grassmannian of 4-dimensional subspaces of $T_u \F$ except
at points where $\lambda=\nu=\alpha/2$ and $\mu=0$.  In that case \eqref{inteltq}
implies that $\alpha^2+4c=0$.  Thus, our assumption about $\alpha$ guarantees that the
set of integral 4-planes satisfying the independence condition is a smooth submanifold.
Since the Grassmannian has dimension 16, the submanifold has codimension 14.

Let $u = \Gamma(t)$ for an arbitrary $t$. Because $\bv=\Gamma'(t)$ is non-characteristic and $\I$ has
no algebraic generators of degree higher than two, $E=H(\bv)$ is
the unique integral 4-plane containing $\Gamma'(t)$.

Cartan's Test for involutivity
can be formulated in terms of the codimensions of the polar spaces for a flag
of integral elements terminating in $E$.  To this end, let
$E_0 = \{0\}$, $E_1 = \{\bv\}$, let $E_2 \subset E_3$ be any 2- and 3-dimensional
integral elements contained in $E$, and let $c_k$ denote the codimension
of $H(E_k)$ for $k=0, \ldots 3$.  Then $c_0=2$ (because there are only two 1-forms in the
ideal), $c_1=4$ as computed above, and $c_2=c_3=4$ because $\I$ has
no additional algebraic generators.  Then, because
$c_0+ c_1+c_2+c_3=14$
coincides with the codimension of the set of integral 4-planes, $\I$
is involutive and in particular the members of the flag are K\"ahler-regular integral elements.
Successive applications of the Cartan-K\"ahler Theorem give the existence and
uniqueness of $\Sigma$ containing $\Gamma$.
\end{proof}

Note that the image under $\pi:\F \to \Mt$ of the four-dimensional integral
manifold constructed in Theorem \ref{applyCK} is a three-dimensional Hopf hypersurface.
In the remainder of this section we will solve a geometric initial value problem
for such hypersurfaces.

\begin{proof}[Proof of Theorem \ref{HopfIVP}]
Let $\Gamma:I\to \F$ be the lift of $\gamma$ provided by
the Frenet frame vectors satisfying \eqref{Frenet}, with $e_1=T$, $e_2=\JJ T$, $e_4=N$
and $e_3=-\JJ N$.  We will first show that
$\Gamma$ satisfies the conditions of Proposition \ref{applyCK}.

Because the frame vectors $e_2,e_3,e_4$ are orthogonal to $\gamma'$,
we have $\Gamma^* \w^2=\Gamma^*\w^3=\Gamma^*\w^4=0$.  Next, the Frenet equations \eqref{Frenet}
imply that
\begin{equation}\label{estar}
\Gamma^*\w^4_1=\kappa_1\,ds, \qquad \Gamma^* \w^4_2=\Gamma^*\w^4_3=0
\end{equation}
  (Here, $s$ is
an arclength coordinate along $\gamma$.)
In particular, $\Gamma$ is an integral curve of the 1-forms $\theta_1=\w^4$ and
$\theta_2=\w^4_3 -\alpha \w^3$.  If we set $\bv = \Gamma'(s)$ in \eqref{vintprod}
then $a=1$, $b=0$, $p=\kappa_1$, $q=0$, and the characteristic equations
\eqref{CV} take the form $0=2\kappa_1 -\alpha = \kappa_1^2+c$, which
cannot simultaneously hold because of our assumption that $\alpha^2 +4c \ne0$.
Thus, $\Gamma$ is not characteristic, and by Proposition \ref{applyCK} there exists a unique integral
manifold $\Sigma^4$ passing through $\Gamma$.  Then $M=\pi(\Sigma)$ is a
Hopf hypersurface containing $\gamma$.  Because $\Gamma^*\w^3=0$,
$\gamma$ is tangent to the holomorphic distribution on $M$.
Moreover, \eqref{estar} shows that $\nat_{e_1} e_4 = -\kappa_1 e_1$
along $\gamma$, and thus
$\gamma$ is a principal curve in $M$ with principal curvature $\kappa_1$,

Conversely, suppose $M$ is a Hopf hypersurface containing $\gamma$,
in which $\gamma$ is tangent to the holomorphic distribution.
Then there exists a unitary frame along $M$, in an open neighborhood
of $\gamma$, such that $e_1$ is tangent to $\gamma$ and $e_3$ is the structure
vector.  Moreover, if $\gamma$ is principal in $M$, then
$\nat_{\gamma'} e_4$ must be a multiple of $e_1$.  Thus,
the covariant derivatives of the frame vectors with respect to $\gamma'$
satisfy the Frenet equations with $\tau=0$.  Hence, the unitary
frame constructed along $M$, when viewed as a submanifold of $F$, passes through the curve $\Gamma$ constructed above,
and $M$ must be the image of the unique integral manifold $\Sigma^4$ through $\Gamma$.
\end{proof}

\subsection{Non-Hopf hypersurfaces}
In this section we will investigate the two possible kinds of shape operator, given by Proposition \ref{etcProp},
for non-Hopf hypersurfaces in $\Mt=\CP^2$ or $\CH^2$ that are both *-Einstein and pseudo-Ryan.  The two lemmas given in this
section will prove Proposition \ref{nonHopfpseudoR}.  For a more detailed explanation of the method of proof used here,
see \S 6 of \cite{IveyRyan2}.

\begin{lemma} There are no hypersurfaces in $\Mt$ that satisfy condition (i) in Proposition \ref{etcProp}.
\end{lemma}
\begin{proof}  Using \eqref{omegaH} and \eqref{Hmatrix}, we see that adapted frames along
such a hypersurface correspond to integral 3-manifolds of the Pfaffian system $\I$ generated
differentially by the 1-forms
$$\theta_0 = \w^4, \qquad
\theta_1 = \w^4_1 - \beta \w^3, \qquad
\theta_2 = \w^4_2 - \nu \w^2, \qquad
\theta_3 = \w^4_3 - \beta \w^1 - \alpha \w^3.
$$
We define this exterior differential system on $\F \times \R^2$,
with $\beta,\nu$ used as coordinates on the second factor, and $\alpha$
given in terms of these by solving \eqref{firstoption}:
$$\alpha = \dfrac{\beta^2(\nu^2 - 4c)}{4c \nu}.$$
We restrict to the open subset on which $\beta$ and $\nu$ are
both nonzero, and take the usual independence condition.

We compute the exterior derivatives of these generator
1-forms modulo themselves.  As usual,
$d\theta_0 \equiv 0$ modulo $\theta_0, \ldots, \theta_3$,
while
\begin{equation}\label{case1tableau}
d\begin{bmatrix} \theta_1 \\ \theta_2 \\ \theta_3+\dfrac{\beta}{\nu}\theta_1 \end{bmatrix}
\equiv -\begin{bmatrix}
0 & \pi_1 & \pi_2 \\
\pi_1 & \pi_3 & -(\beta/\nu) \pi_1 \\
\pi_2 & 0  &
\dfrac{\beta}{4c \nu^2}\left(2\nu(\nu^2-2c)\pi_2 + \beta(\nu^2+4c)\pi_3\right)
\end{bmatrix}
\& \begin{bmatrix}\w^1 \\ \w^2 \\ \w^3 \end{bmatrix}
\tmod \theta_0, \ldots, \theta_3,
\end{equation}
where
\begin{align*}
\pi_1 &= -\nu\, \w^2_1 + 2\beta\nu\,\w^1-(\beta^2+c)\w^3,\\
\pi_2 &= d\beta - \dfrac{\beta^2(\nu^2+4c)+8c^2}{4c}\,\w^2,\\
\pi_3 &= d\nu +\dfrac{\nu(\beta^2\nu^2(\nu^2-2c) + 8c^2(4\nu^2 - 3c))}{2c\beta(\nu^2+4c)}\,\w^2.
\end{align*}
In this computation, we further restrict to the open
subset where $\nu^2+4c\ne 0$.  (For any solution, this
condition will either hold on an open subset, or $\nu$
will be locally constant; we will consider the latter possibility
below.)

Inspecting the generator 2-forms given by \eqref{case1tableau}
shows that, on any solution, $\pi_2$ will be a multiple of $\w^3$,
$\pi_1$ will be a multiple of $\w^2$, and
$\pi_3$ will be a multiple of $\w^1-(\beta/\nu)\w^3$,
and furthermore each of these multiples determines the others.
More precisely, there will be a function $\rho$ such that
the following 1-forms vanish:
\begin{align*}\theta_4 &= \pi_1 - \rho\,\w^2,\\
\theta_5 &=\pi_2 - \dfrac{\rho\beta^2(\nu^2+4c)}{4c \nu^2}\,\w^3,\\
\theta_6 &=\pi_3 - \rho(\w^1 -(\beta/\nu)\w^3).
\end{align*}
To solve for this function, we add $\rho$ as a new coordinate,
and define the 1-forms $\theta_4, \theta_5, \theta_6$
on the open subset of $\F\times \R^3$ where the
nonzero conditions on $\beta,\nu$ hold.  The Pfaffian
system generated by $\theta_0,\ldots,\theta_6$ is the prolongation of
$\I$.

We compute the exterior derivatives of the new 1-forms
of the prolongation.  In particular, we find
\begin{align*}
d \theta_4 \& \w^2 &\equiv \rho \dfrac{(\beta^2\nu^2-4c^2)}{2c\nu}
\,\w^1 \& \w^2 \& \w^3,\\
d \theta_6 \& (\w^1 -(\beta/\nu)\w^3) &\equiv
\rho \dfrac{\beta^2\nu^4(8c-\nu^2) + 8c^2(3\nu^4+42c\nu^2-32c^2)}{8c^2\nu (\nu^2+4c)}
\,\w^1 \& \w^2 \& \w^3
\end{align*}
modulo $\theta_0, \ldots, \theta_6$.
Because of our independence condition, at any point of $M$ either $\rho$ vanishes
or both polynomials in $\beta$ and $\nu$ in the numerators on the right-hand side vanish;
moreover, one alternative or the other must hold on an
open subset of $M$.  Note that in the latter case
$\beta$ and $\nu$ must be locally constant.

Consider first the case where $\rho$ vanishes identically
on an open subset.  Then the 1-forms $\pi_1, \pi_2, \pi_3$
all vanish, as do their exterior derivatives.  We compute
\begin{align*}
\w^3 \& d\pi_1 &\equiv -\dfrac{\nu}{2c}
\left[\beta^2(\nu^2-2c)+2c^2\right]\,\w^1 \& \w^2 \& \w^3,\\
d\pi_2 &\equiv \dfrac{\beta^2+c}{4c\nu}\left[ \beta^2(\nu^2+4c)+8c^2\right]\,\w^1 \& \w^3
\end{align*}
modulo $\theta_0,\ldots,\theta_3,\pi_1, \pi_2, \pi_3$.
The vanishing of the expressions on the right of both equations implies that $\beta^2+c=\nu^2-4c=0$,
which is impossible.

Thus we may work in a small open set where $\rho\ne 0$ and so $\beta$, $\nu$, and $\alpha$ are constant.
We restrict $\I$ to a submanifold where $\beta$ and $\nu$ are nonzero constants and compute
$$d\left(\theta_3+\dfrac{\beta}{\nu}\theta_1\right)
\equiv -\dfrac{\beta^2(\nu^2+4c)+8c^2}{4c} \w^1 \& \w^2
+ \beta\, \dfrac{\beta^2(\nu^2-2c)+2c^2-6c\nu^2}{2c\nu}\w^2 \& \w^3
$$
modulo $\theta_0, \ldots,\theta_3$.  It is easy to check that the numerators of the terms
on the right cannot simultaneously vanish.  This is a contradiction.
\end{proof}

\begin{lemma}
Any hypersurface in $\Mt$ that satisfies condition (ii) in Proposition \ref{etcProp}
must have $\alpha, \beta, \lambda$ locally constant.
\end{lemma}
\begin{proof} Let $U\subset \F\times \R^3$ be the open subset
where the coordinates $\alpha, \beta, \nu$ on the second factor satisfy
$\beta \ne 0$ and $\nu \ne0$.
On $U$, let $\I$ be the Pfaffian system generated by the 1-forms
$$\theta_0 = \w^4, \quad \theta_1 = \w^4_1+(\sigma/\nu)\w^1 - \beta \w^3, \quad
\theta_2 = \w^4_2 - \nu \w^2, \quad
\theta_3 = \w^4_3 - \beta \w^1 - \alpha \w^3,$$
for a nonzero constant $\sigma$.
Then an adapted framing along a hypersurface $M$ satisfying condition (ii), for the given $\sigma$,
generates an integral submanifold $\Sigma^3$ of $\I$.  Moreover, $\Sigma$ will lie inside
the submanifold $V \subset U$ determined by imposing \eqref{secondoption} on the coordinates.

We will first examine the structure equations of $\I$ on $U$, later passing to
the restriction of $\I$ to V.
We assume that $\nu^2+4c \ne 0$ and $\nu^2+\sigma\ne 0$ on an open subset of $\Sigma$; we will address the case where $\nu$
is locally constant later.  (Note that $\nu^2+\sigma$ cannot vanish on $V$, since substituting
$\nu^2=-\sigma$ in \eqref{secondoption} implies that $\beta=0$.)

The exterior derivatives of the generators of $\I$ satisfy $d\theta_0 \equiv 0$ and
\begin{equation}\label{case2tableau}
d\begin{bmatrix} \theta_1 \\ \theta_2 \\ \theta_3 \end{bmatrix}
\equiv -\begin{bmatrix}
\dfrac{\sigma}{\nu^2} \pi_1 & -\dfrac{\nu^2+\sigma}{\beta\nu}\pi_2 & \pi_3 \\
-\dfrac{\nu^2+\sigma}{\beta\nu}\pi_2 & \pi_1 & \pi_2 \\
\pi_3 & \pi_2  & \pi_4
\end{bmatrix}
\& \begin{bmatrix}\w^1 \\ \w^2 \\ \w^3 \end{bmatrix}
\tmod \theta_0, \ldots, \theta_3,
\end{equation}
where
\begin{align*}
\pi_1 &= d\nu -\dfrac{\beta^2\nu(\nu^2-2\sigma)+(\nu^2+\sigma)((c-Z)\nu + \sigma(\nu-\alpha))}{\sigma \beta}\w^2,\\
\pi_2 &= \beta \w^2_1 + \dfrac{(\beta^2-c-Z)\nu+\sigma(\nu-\alpha)}{\nu}\w^1+Z\dfrac{\beta\nu}{\nu^2+\sigma}\w^3,\\
\pi_3 &= d\beta -(\beta^2+\alpha\nu+c+\sigma+Z)\w^2,\\
\pi_4 &= d\alpha + \left(\beta(3\nu-\alpha)+Z\dfrac{\beta\nu}{\nu^2+\sigma}\right)\w^2
\end{align*}
and
$$Z =\dfrac{4c\beta^2(\nu^4+4(c-\sigma)\nu^2-4c\sigma+\sigma^2)}{(\nu^2+\sigma)(\nu^2+4c)(4c-\sigma)}+\dfrac{(c+\sigma)\nu^4+(4c^2+20c\sigma+\sigma^2)\nu^2+3c\sigma(\sigma-4c)}{\nu^2(\nu^2+4c)}.$$
The structure equations \eqref{case2tableau} show that
there is a 4-parameter family of 3-dimensional integral elements (satisfying the independence condition)
at every point of $U$.  However, we will only consider those integral elements that are
tangent to $V$.

When restricted to $V$, the 1-forms $\pi_1,\ldots,\pi_4$ are no longer linearly
independent.  In fact, they satisfy a homogeneous linear relation
\begin{equation}\label{locusrel}
(P \pi_1 + Q \pi_3 + R \pi_4)\restr_V = 0,
\end{equation}
where
$$P = \dfrac{\beta^2(\nu^4+(4c+2\sigma)\nu^2+\sigma^2-4c\sigma)}{\nu(\nu^2+\sigma)},\quad
Q=2\beta(\nu^2+\sigma-4c), \quad R= \dfrac{(\sigma-4c)(\nu^2+\sigma)}{\nu}.$$
(The value of $Z$ is chosen so as to make the right-hand side of \eqref{locusrel} equal
to zero.)

Because the pullbacks of the 2-forms in \eqref{case2tableau} to $\Sigma \subset V$ must
vanish, and the restrictions of the $\pi_i$ to $V$ satisfy \eqref{locusrel}, the pullbacks
of the $\pi_i$ to $\Sigma$ are determined up to multiple.
That is, there exists a function $\rho$ on $\Sigma$
such that the pullbacks of these forms to $\Sigma$ satisfy
\begin{align}
\pi_1 &= \rho \dfrac{\nu^2}{\sigma}\left(\w^3 - \dfrac{\nu^2+\sigma}{\beta\nu}\w^1\right),
\label{denument}
\\
\pi_2 &= \rho\dfrac{\nu^2}{\sigma}\w^2,
\label{dew21ment} \\
\pi_3 &= \rho (\w^1 + S \w^3),\\
\pi_4 &= \rho (S\w^1 + T \w^3),
\end{align}
where $S,T$ are determined by substitution in \eqref{locusrel}, i.e.,
$$ -\dfrac{\nu(\nu^2+\sigma)}{\sigma\beta} P + Q + RS = 0,
\quad \dfrac{\nu^2}{\sigma}+Q S + R T = 0.
$$
Note that, from now on, we will be working on $V$, taking $\alpha$ to be
given by solving \eqref{secondoption}, i.e.,
\begin{equation}\label{avalue}
\alpha = \dfrac{\beta^2\nu(\nu^2+\sigma-4c)}{(4c-\sigma)(\nu^2+\sigma)}.
\end{equation}

Differentiating both sides of \eqref{denument} and wedging with $\w^3 - \dfrac{\nu^2+\sigma}{\beta\nu}\w^1$ yields the integrability condition
\begin{multline*}[
{\nu}^{4}(
{\nu}^{4}+(2\sigma-8c){\nu}^{2}-{\sigma}^{2}) {\beta}^{2}\\
-(\sigma+{\nu}^{2})
(4c-\sigma)\left(
 (\sigma+6c){\nu}^{4}
  +(84c^2-7c\sigma-\sigma^2){\nu}^{2}
  -64c^{3}+28\sigma c^{2}-3c{\sigma}^{2}\right)]\rho=0.
\end{multline*}
Similarly, differentiating both sides of \eqref{dew21ment} and wedging with $\w^2$ yields
the integrability condition
\begin{multline*}
[4{\nu}^{2} \left( {\nu}^{6}+(4c-\sigma){\nu}^{4}
+(8c\sigma-3{\sigma}^{2}){\nu}^{2}-4{\sigma}^{2}c+{\sigma}^{3}
 \right) {\beta}^{2}
 \\
 - ( \sigma+{\nu}^{2})( 4c-\sigma)  \left(
(9\sigma+4c){\nu}^{4}
+(16{c}^{2}+92c\sigma-14{\sigma}^{2})\nu^2
 +16{\sigma}^{2}c+{\sigma}^{3} - 80\sigma{c}^{2}
\right)]\rho =0.
\end{multline*}
Thus, either $\rho=0$ on an open set in $\Sigma$, or $\nu$ is locally constant.

Suppose $\rho=0$.  Then the 1-forms $\pi_1, \ldots, \pi_4$ vanish on $\Sigma$,
and we may derive additional integrability conditions as follows.
By computing $d\pi_1$ and $d\pi_3$ modulo $\theta_0, \ldots, \theta_3$
and $\pi_1, \ldots, \pi_4$, we obtain
\begin{gather}
[{\nu}^{4}\left( 2\nu^2-4c-\sigma \right) {\beta}^{2} + \left(
\sigma+{\nu}^{2} \right)  \left( 4c-\sigma \right)
\left((\sigma+16c)\nu^2- 12{c}^{2}+3c\sigma\right)]
W=0,\label{ginger}\\[6pt]
[(4c-\sigma)(\sigma+{\nu}^{2})\left(2(\sigma+c)\nu^4+(\sigma^2+24c\sigma+8c^2)\nu^2
+3c{\sigma}^{2}-12{c}^{2}\sigma\right )
\label{fred} \\
+
{\nu}^{4}\left(\nu^4+8c\nu^2-{\sigma}^{2}-12c\sigma+16{c}^{2}\right){\beta}^{2}
] W =0,\notag
\end{gather}
where
$$W = 4{\nu}^{2} \left( {\nu}^{4}+4(c-\sigma){\nu}^{2}+{\sigma}^{2}-4c\sigma \right) {\beta}^{2}+
(4c-\sigma)(\nu^2+ \sigma)
\left(\nu^4+(4c+16\sigma)\nu^2 -12c\sigma-{\sigma}^{2} \right).$$
Thus, either the polynomial $W$ vanishes, or else both
polynomials in square brackets in \eqref{ginger},\eqref{fred} vanish.
In the latter case, taking resultants with respect to $\beta$ shows
that $\nu$ must be locally constant.  If $W$ vanishes on an open set,
then solving for $\beta$ and differentiating $W=0$ modulo
the $\theta_i$ and $\pi_j$ yields another polynomial in $\nu$ which
must vanish.  Thus, again we conclude that $\nu$ must be locally constant.

Finally we reconsider the original system $\I$ restricted to a submanifold of
$V$ on which $\nu$ is equal to a nonzero constant, and hence $d\nu=0$.
Differentiating the 1-forms of $\I$ reveals an additional integrability condition,
as follows.  We compute that
\begin{multline*}
d\left(\theta_1 + \dfrac{\nu^2+\sigma}{\beta\nu}\theta_3\right)
\& \left( \dfrac{\nu^2+\sigma}{\beta\nu}\w^1-
\dfrac{2\nu^2 + \sigma-4c}{\sigma-4c}\w^3\right)\\
= \dfrac{{\nu}^{4}\left (-2{\nu}^{2}+\sigma+4c\right ){\beta}^{2}+
(\sigma-4c)(\nu^2+\sigma)\left(\sigma{\nu}^
{2}+3c\sigma-12{c}^{2}+16c{\nu}^{2}\right)
}{\beta\nu(\sigma-4c)^2}
\,\w^1 \& \w^2 \& \w^3.
\end{multline*}
The polynomial in the numerator on the left must vanish.
Because this cannot happen if the coefficient of $\beta^2$ also vanishes,
we conclude that $\beta$ is locally constant on solutions.
It follows that $\alpha$, given by \eqref{avalue},
and $\lambda=-\sigma/\nu$ are also locally constant.
\end{proof}
We note that, by doing further computations with this exterior
differential system, one can show that no solutions with
$\beta$ and $\nu$ both constant exist.

\subsection{Non-Hopf examples with restricted shape operator}\label{coho}

We now construct an interesting class of
non-Hopf hypersurfaces $M$ in $\CP^2$ and $\CH^2$,
obtained by solving a certain underdetermined system of ordinary differential equations.
In particular, this will show the existence of non-Hopf pseudo-Ryan hypersurfaces
(see Theorem \ref{nonempty} and Corollary \ref{pseudoRexist}).

Let $M$ be a hypersurface in $\CH^n$ or $\CP^n$, with structure vector field $W$.
At each $p\in M$ we define the subspace $\H_p \subset T_pM$ as
the smallest subspace that contains $W$ and is invariant under the
shape operator $A$.  Then $M$ is Hopf if an only if
$\H_p$ is one-dimensional at each point.  In what follows, we
restrict to the case $n=2$, and consider
those hypersurfaces $M$ where $\H$ is a smooth two-dimensional distribution on $M$.
This means that we can locally construct an adapted
orthonormal frame $(W,X,\varphi X)$ with respect to which the shape operator
has the form
\begin{equation}\label{Ah2form}
A = \begin{pmatrix} \alpha & \beta & 0 \\
\beta & \lambda & 0 \\
0 & 0 & \nu \end{pmatrix},\end{equation}
and $\H$ is spanned by $W$ and $X$ at each point.  Note that $Y=\varphi X$ is
thus a principal vector.
Our next result shows that it is relatively easy to generate examples of such hypersurfaces.

\begin{theorem} \label{construction}
 Let $\alpha(t), \beta(t), \lambda(t), \nu(t)$ be analytic functions on an
open interval $I\subset \R$ satisfying the underdetermined ODE system
\begin{equation}\label{odesys}
\begin{aligned}
\alpha' &= \beta(\alpha+\lambda-3\nu), \\
\beta' &= \beta^2 + \lambda^2 - 2\lambda\nu + \alpha \nu+c, \\
\lambda' &= \left(\dfrac{(2\lambda+\nu)\beta^2 +(\nu-\lambda)(\alpha\lambda-\lambda^2 +c)}{\beta}\right),
\end{aligned}\end{equation}
with $\beta(t)$ nowhere zero.
Let $\gamma(t)$ be a unit-speed analytic framed curve in $\Mt$, defined for $t \in I$,
with transverse curvature $\nu(t)$ and zero holomorphic curvature and zero torsion.
Then there exists a non-Hopf hypersurface $M^3$ such that

(i) the distribution $\H$ is rank 2 and integrable;

(ii) $M$ has a globally defined
frame $(W,X,\varphi X)$ with respect to which the shape operator has the
form \eqref{Ah2form}, such that $\alpha, \beta, \lambda$ and $\nu$
are constant along the leaves of $\H$, and

(iii) $M$ contains $\gamma$ as a principal curve
to which the vector field $Y=\varphi X$ is tangent, and along which
the components of $A$ restrict to coincide with the
given solution of the ODE system.
\end{theorem}

\begin{proof}  On $F \times \R^4$, with $\alpha, \beta, \lambda, \nu$ as
coordinates on the second factor, define the 1-forms
$$\theta_0 = \w^4, \quad
\theta_1 = \w^4_1 - \lambda \w^1 - \beta\w^3, \quad
\theta_2 = \w^4_2 - \nu \w^2, \quad
\theta_3 = \w^4_3 - \beta \w^1 - \alpha \w^3,
$$
where the $\w^i$ and $\w^i_j$ are pulled back to the product manifold
via projection to the first factor.
(We will restrict to the open subset of $\F \times \R^4$ where $\beta \ne 0$.)
Then a non-Hopf hypersurface $M$ equipped with an orthonormal frame with respect
to which the shape operator has the form \eqref{Ah2form} can be lifted to a three-dimensional integral manifold $f(M)$
of these 1-forms, by letting $e_1 = X$, $e_2 = Y$, $e_3=W$, $e_4=\xi$,
and letting the coordinates $\alpha, \beta, \lambda, \nu$ take the values
of the corresponding components of $A$.  (Note that this integral manifold
also satisfies the usual independence condition $\w^1 \& \w^2 \& \w^3 \ne 0$.)
In what follows, we will derive necessary conditions that this integral manifold
must satisfy, if $M$ is to satisfy the conditions (i) and (ii) of the theorem.

If $\H$ is integrable then $f^*(d\w^2 \& \w^2)=0$.  We compute
$$d\w^2 \& \w^2 \equiv (-\w^2_1+\lambda \w^3) \& \w^1 \& \w^2 \tmod \theta_0, \theta_1, \theta_2, \theta_3.$$
If $\nu$ is constant along the integral surfaces of $\H$, then $f^*(d\nu \& \w^2)=0$.
We compute
$$d\theta_2 \equiv (\nu-\lambda) \w^2_1 \& \w^1 -\beta \w^2_1 \& \w^3 +(\beta^2 -\lambda(\alpha-\nu)-c) \w^1 \& \w^3 \tmod \theta_0, \ldots, \theta_3, d\nu \& \w^2.$$
So, the last two conditions imply that $f(M)$ is also an integral of the 1-form
$$\theta_4 = \w^2_1 - \lambda \w^3 - \left( \dfrac{\beta^2+\lambda^2 -\alpha \lambda-c}{\beta} \right)\w^1,$$
Now we compute
$$d\theta_3 \equiv \w^1 \& (d \beta - (\beta^2 + \lambda^2 - 2\lambda\nu + \alpha \nu+c)\w^2) +
\w^3 \& (d \alpha - \beta(\alpha+\lambda-3\nu)\w^2)\tmod \theta_0, \ldots, \theta_4.$$
Thus, the condition that $\alpha, \beta$ have nonzero derivatives only in the
$Y$-direction implies that $f(M)$ is also an integral of the 1-forms
$$\theta_5 = d \alpha - \beta(\alpha+\lambda-3\nu)\w^2,\qquad
\theta_6 = d \beta - (\beta^2 + \lambda^2 - 2\lambda\nu + \alpha \nu+c)\w^2.$$
Similarly, computing $d\theta_1$ modulo $\theta_0, \ldots, \theta_6$, and using
the condition that $\lambda$ has a nonzero derivative only in the $Y$-direction shows that
$f(M)$ is also an integral of
$$\theta_7 = d \lambda - \left(\dfrac{(2\lambda+\nu)\beta^2 +(\nu-\lambda)(\alpha\lambda-\lambda^2 +c)}{\beta}\right)\w^2.$$

In order to encode the condition that $f^*(d\nu \& \w^2)=0$, we introduce a new coordinate
$p$ and define the 1-form
$$\theta_8 = d \nu - p\w^2.$$
This, and the previous 1-forms $\theta_i$,
are taken to be defined on the open set in $F \times \R^5$ where $\beta\ne 0$.
The framed hypersurfaces satisfying the conditions in the theorem are in one-to-one
correspondence with integral manifolds (satisfying the independence condition) of the Pfaffian system
$\I$ defined by $\theta_0, \ldots, \theta_8$.

It is now easy to verify that this exterior differential system is involutive, with its only nonzero
Cartan character being $s_1=1$.  Moreover, the Cartan-Kahler Theorem implies that
integral manifolds exists that pass through any non-characteristic 1-dimensional
integral manifold of $\I$.  In particular, any integral curve $\Gamma$ along which
$\w^1=\w^3=0$ but $\w^2\ne 0$ is non-characteristic.  We will now show how
such a curve corresponds exactly to a curve $\gamma$ in $\Mt$ satisfying
the conditions in Theorem \ref{construction}.

Given $\gamma$, equipped with a unitary frame satisfying
the Frenet equations \eqref{Frenet}, we construct a lift $\widehat\gamma$ into $\F$ by
setting $e_2=T$, $e_1=-\JJ T$, $e_4=N$, $e_3 = -\JJ N$.  It follows that $\w^1, \w^3, \w^4$ and $\w^4_1=-\w^3_2$ pull back
to be zero along $\widehat\gamma$, and $\w^2_1$, $\w^4_2$ and $\w^3_4$
pull back to be multiples of $\w^2$ that respectively
are the holomorphic curvature, transverse curvature and torsion of $\gamma$.
Thus, if $\gamma$ has zero holomorphic curvature then $\widehat\gamma$ is
an integral curve of $\w^2_1$.  We further lift the curve into
$\F \times \R^5$ by setting $\alpha, \beta,\lambda,\nu$ equal
to the values given by the solution to the ODE system,
and $p$ equal to $d\nu/dt$.  Then it is easy to check that
the lifted curve $\Gamma$ is an integral curve of $\theta_0, \ldots, \theta_8$.
\end{proof}

\begin{corollary}\label{pseudoRexist}
Let $\alpha(t),\beta(t),\lambda(t),\nu(t)$ be analytic solutions defined for $t\in I$ of the
system \eqref{odesys}, such that $\beta$ is nowhere zero and
$$\beta^2 \nu^2 + (4c+\lambda \nu)(\alpha(\lambda-\nu)-\beta^2)=0.$$
Then the hypersurface $M$ constructed by the previous theorem is
a non-Hopf pseudo-Ryan hypersurface.
\end{corollary}

Similarly, we can use the above theorem, together with
solutions to the ODE systems, to construct non-Hopf hypersurfaces
satisfying $\mu=0$ and any given algebraic condition involving $\alpha, \beta, \lambda$ and $\nu$.

\subsection{Non-Hopf hypersurfaces with constant principal curvatures}\label{constantpc}

Theorem \ref{construction} provides a new construction for the
non-Hopf hypersufaces in $\CH^2$ with constant principal curvatures.
These have been classified by Berndt and Diaz-Ramos \cite {BerndtDiaz},
who showed that such hypersurfaces must be open subsets of
homogeneous hypersurfaces.  Thus, they belong to a 1-parameter family of orbits under the action of a certain 3-dimensional group
of isometries of $\CH^2$.
One member of the family is a minimal hypersurface and the others are its equidistant
hypersurfaces.

In Theorem \ref{construction}, we take $\nu$ to be any constant in the range
$-1/r < \nu < 1/r$ and solve \eqref{odesys} for a constant solution.
(In fact, a constant solution is possible only when $\nu$ lies in
this range.)  The shape operator can be written with respect to the frame $(W, X, \varphi X)$, used in \S\ref{RicciConditions}, as

\begin{equation}
A = \frac{1}{r}
\begin{pmatrix} 3 u - u^3& v& 0 \\
                 v &u^3 & 0 \\
                            0&0&u
\end{pmatrix}
\end{equation}
where $u = r \nu$ and $v = (1- u^2)^\frac{3}{2}.$  The principal curvatures are $\nu$ and

$$ \frac{3}{2}\ \nu \pm \frac{1}{r}\ \sqrt{ 1 - \frac{3}{4}\ r^2 \nu^2}.$$
\noindent
On setting $r = 2$, we see that our result is consistent with Proposition 3.5 of \cite{BerndtDiaz}.

\end{document}